\documentclass[12pt]{amsart}
\usepackage[utf8]{inputenc}
\usepackage{amsmath}
\usepackage[utf8]{inputenc}
\usepackage{amssymb}
\usepackage{amsthm}
\usepackage{graphicx}
\usepackage{color, soul}
\usepackage{centernot}
\usepackage{verbatim}
\usepackage{multirow}
\usepackage{array}   
\newcolumntype{L}{>{$}l<{$}} 
\usepackage{setspace}
\usepackage{mathrsfs}
\usepackage{enumitem}
\usepackage{tikz}
\usepackage{bigints}
\usepackage{bm}
\usepackage{caption}
\usepackage{subcaption}
\usepackage{url}
\usepackage{imakeidx}
\makeindex[columns=2, intoc]
\usepackage{tikz-cd}

\makeatletter
\renewcommand{\@biblabel}[1]{[#1]\hfill}
\makeatother

\newtheorem{theorem}{Theorem}[section]
\newtheorem{corollary}[theorem]{Corollary}
\newtheorem{lemma}[theorem]{Lemma}
\newtheorem{proposition}[theorem]{Proposition}

\theoremstyle{definition}
\newtheorem{definition}[theorem]{Definition}

\newtheorem*{notation}{Notation}

\newtheorem{examplex}[theorem]{Example}
\newenvironment{example}
  {\pushQED{\qed}\examplex}
  {\popQED\endexamplex}
\newtheorem{remarkx}[theorem]{Remark}
\newenvironment{remark}
  {\pushQED{\qed}\remarkx}
  {\popQED\endremarkx}

\theoremstyle{remark}

\usepackage{mathtools}
\mathtoolsset{showonlyrefs}

\newcommand{\eps}{\varepsilon}

\newcommand{\QQ}{\mathbb{Q}}
\newcommand{\ZZ}{\mathbb{Z}}

\newcommand{\FF}{\mathbb{F}}

\newcommand{\Proj}{\mathbb{P}}
\DeclareMathOperator{\Tr}{Tr}

\DeclareMathOperator{\Res}{Res}

\DeclareMathOperator{\Gal}{Gal}

\DeclareMathOperator{\Hom}{Hom}

\DeclareMathOperator{\Ind}{Ind}
\DeclareMathOperator{\End}{End}

\DeclareMathOperator{\disc}{disc}
\DeclareMathOperator{\Prym}{Prym}

\DeclareMathOperator{\wild}{wild}

\DeclareFontFamily{U}{wncy}{}
\DeclareFontShape{U}{wncy}{m}{n}{<->wncyr10}{}
\DeclareSymbolFont{mcy}{U}{wncy}{m}{n}
\DeclareMathSymbol{\Sha}{\mathord}{mcy}{"58}

\newcommand{\PP}{\mathfrak{P}}

\renewcommand{\div}{\mathrm{div}}

\renewcommand{\bar}{\overline}
\renewcommand{\tilde}{\widetilde}

\usepackage{geometry}
\usepackage{fancyhdr}
\usepackage[hidelinks]{hyperref}
\usepackage[T1]{fontenc}
\usepackage{longtable}
\usepackage{dsfont}

\newcommand{\nocontentsline}[3]{}
\let\origcontentsline\addcontentsline
\newcommand\stoptoc{\let\addcontentsline\nocontentsline}
\newcommand\resumetoc{\let\addcontentsline\origcontentsline}
\newcommand{\annot}[1]{\text{(\small#1)}}

\DeclareMathOperator{\one}{\mathds{1}}

\title{Wild conductor exponents of curves}

\author{Harry Spencer}
\address{University College, London, WC1H 0AY, UK}
\email{harry.spencer.22@ucl.ac.uk}

\begin{document}

\begin{abstract}
    We give an explicit formula for wild conductor exponents of plane curves over $\QQ_p$ in terms of standard invariants of explicit extensions of $\QQ_p$, generalising a formula for hyperelliptic curves. To do so, we prove a general result relating the wild conductor exponent of a simply branched cover of the projective line with its associated discriminant cover. In an appendix, we resolve a minor issue in the literature on the $3$-torsion of genus 2 curves.
\end{abstract}

\maketitle

\tableofcontents 


\section{Introduction}

Associated to a curve ${C}$ over a finite extension $K$ of $\QQ_p$, is a representation-theoretic invariant which measures bad reduction called the \textit{(local) conductor}. The conductor is an ideal $N=(\pi^{n_{{C}}})$, where $\pi$ is a uniformiser, and $n_C$ is the \textit{conductor exponent}. In turn, $n_{C}$ is defined to be
\[n_{{C}} = n_{{C},\text{tame}} + n_{{C},\wild},\]
where the \textit{tame conductor exponent}, $n_{{C},\text{tame}}$, can be extracted from a regular model of ${C}/K$. The \textit{wild conductor exponent} is, in general, difficult to determine but is 0 if $p>2g_{C}+1$, where $g_{C}$ is the genus of ${C}$. It is derived from the Galois-module (in fact, wild inertia-module) structure of the $\ell$-torsion $J_{C}[l]$ on the Jacobian of ${C}/K$ for any $p\ne\ell$. See \S\ref{sec:conductors} for the key points.

The problem of provably (and efficiently) finding conductors of elliptic curves is solved by Tate's algorithm, see \cite[Chapter 4]{SilvermanAdvanced}. For hyperelliptic curves the problem is solved for $p\ne2$ by means of an explicit formula in \cite[Theorem 11.3]{M2D2}. In \cite{Lupoian2022} the problem is solved for $p\ne2$ for curves of genus at most 5. In \cite[\S4]{DokDoris}, by providing an algorithm for the explicit computation of 3-torsion (cf.\ \S\ref{app:A}), the problem is solved for genus 2 curves, and in \cite[\S3]{Lupoian2022_Genus3} the problem is solved for hyperelliptic genus 3 curves analogously. The problem is solved for plane quartics with a rational point in \cite[Theorem 2]{LupoianRawson2024}. A formula is given in \cite[Theorem 4.1.4]{KohlsThesis} for conductor exponents of superelliptic curves of exponent $n$, for $p\nmid n$.

In this work, we show how wild conductor exponents can be determined from the ramification locus of a degree $n$ cover of $\Proj^1$, for $p>n$. We first treat the case of simply branched covers of $\Proj^1$.

\begin{theorem}\label{thm:1}
    Let $C/K$ be a curve over a finite extension of $\QQ_p$ equipped with a simply branched degree $n$ cover of $\Proj^1$, with $p>n$. Any hyperelliptic curve $D$ whose degree $2$ map to $\Proj^1$ has the same branch locus as this cover satisfies
    \[n_{C,\wild}=n_{D,\wild}.\]
\end{theorem}

\begin{remark}
    Every such curve $C/K$ admits a simply branched cover of degree at most $g_C+1$ (possibly after base change to a tame extension of $K$), cf.\ Remark \ref{rmk:FultonVariant}.
\end{remark}

In fact, we prove a stronger result (Theorem \ref{thm:Sn_main}) which replaces $\Proj^1$ by an arbitrary curve $B$, replaces the hyperelliptic curve by a degree $2$ cover of $B$, has an additional summand of $(n-1)\cdot n_{B,\wild}$, allowing a slightly weaker restriction on ramification.

By appealing to \cite[Theorem 11.3]{M2D2} (cf.\ Theorem \ref{thm:M2D2}), which gives an explicit formula for wild conductor exponents of hyperelliptic curves over finite extensions of $\QQ_p$, $p\ne2$, we obtain an analogous formula for some plane curves. In fact, \cite[Theorem 11.3]{M2D2} is the special case $n=2$ of the following theorem, for which we first introduce some important notation.

\begin{notation}
    For $K$ a finite extension of $\QQ_p$ and a polynomial $g\in K[t]$ write
    \[w_K(g)=\sum_{r\in R/G_{K}} m(r)\cdot(v_{K}(\Delta_{K(r)/K})-[K(r):K]+f_{K(r)/K}), \]
    where $R$ denotes the set of $\bar{K}$-roots of $g$, $\Delta_\bullet$ the discriminant, $f_\bullet$ the residue degree, and $m(\bullet)$ the multiplicity of a root.
\end{notation}

\begin{remark}\label{rem:hyperelliptic}
    When $g$ is square-free, $w_K(g)$ is the wild conductor exponent of the hyperelliptic curve $y^2=g(t)$ if $p\ne2$. Therefore, the computation of $w_K(g)$ for general $g$ is essentially already implemented, e.g.\ in Magma \cite{Magma}.
\end{remark}

\begin{theorem}[{=Theorem \ref{thm:main2}}]\label{thm:main}
    Let $\mathcal{C}:f(x,y)=0$ be a smooth affine curve over a finite extension $K$ of $\QQ_p$ with smooth compactification $C$. If $p>\deg_x f$, then
    \[n_{{C},\wild}=w_K(\disc_x f).\]
\end{theorem}

\begin{remark}
    To prove Theorem \ref{thm:main}, we need Theorem \ref{thm:Sn_main}, which allows one branch point above which our cover as in Theorem \ref{thm:1} is not simply branched, assuming that we have $S_n$-Galois closure. Note that a simply branched cover of $\Proj^1$ of degree $n$ must have $S_n$-Galois closure (cf.\ \cite[Corollary 3.2]{BKP2024}), and usually so do curves satisfying this less strict criterion.
\end{remark}

\begin{remark}
    Given a curve $C$ embedded smoothly in $\Proj^n$ equipped with a cover $\pi:C\to \Proj^1$, one could obtain a similar result to Theorem \ref{thm:main}, replacing $\disc_x f$ by
    \[\prod_{t\in\Proj^1}\prod_{P\in\pi^{-1}(t)} (x-t)^{e_P-1},\]
    where $e_P$ denotes the ramification index at $P$, provided that one can perturb the defining equations of $C$ to obtain simply branched covers $\tilde{\pi} : \tilde{C} \to \Proj^1$.
\end{remark}

For example, Theorem \ref{thm:main} gives a simple formula for the wild conductor exponent of some superelliptic curves of exponent $n$ for $p>n$.

\begin{corollary}[{=Corollary \ref{cor:super}}]\label{cor:intro}
    Consider a superelliptic curve ${C}/K: y^n = f(x)$, $f$ square-free, over a finite extension $K$ of $\QQ_p$. If $p>n$, then
    \[n_{{C},\wild}=(n-1)\cdot w_K(f).\]
\end{corollary}

All we are using here about these curves is that they are obviously smooth on this affine chart. In particular, we do not use any of the additional structure afforded to a superelliptic curve by nature of it being a cyclic cover of $\Proj^1$---in fact, using perturbations, we deliberately rid ourselves of this structure. In forthcoming work of Obus and Srinivasan, this structure is harnessed using different techniques to generalise Corollary \ref{cor:super}:

\begin{proposition}\label{prop:super}
    Consider a superelliptic curve ${C}/K: y^n = f(x)$ over a finite extension $K$ of $\QQ_p$. Suppose $f=\prod_{i=1}^rf_i^{d_i}$ with the $f_i$ being distinct irreducible polynomials, each $d_i$ satisfying $1\le d_i\le n-1$, and $\gcd(n,d_1,\hdots,d_r)=1$. If $p\nmid 2n$, then
    \[n_{{C},\wild}=\sum_{i=1}^r(n-\gcd(n,d_i))\cdot w_K(f_i).\]
\end{proposition}

In Remark \ref{rmk:super} we give a brief discussion of how the same result could be obtained from the techniques herein. 

We can also apply Theorem \ref{thm:main} in vastly more general settings; the following example computes the wild conductor exponent of a non-superelliptic curve over $\QQ_7$, using Magma for the numerical computations.

\begin{example}
    Consider the curve $C/\QQ_7: f(x,y)=7x^3y^4+x+y^5+7=0$ of genus $6$. We compute that
    \[n_{C,\wild}=w_{\QQ_7}(-1323y^{18} - 18522y^{13} - 64827y^8 - 28y^4)=14.\]
    As a sanity check, one can also compute
    \[n_{C,\wild}=w_{\QQ_7}(\disc_y f)=14,\]
    verifying the (\textit{a priori} non-obvious) fact that the quantity $w_K(\disc_x f)$ is independent of the labelling of $x$ and $y$, provided that the residue characteristic is greater than $\max\{\deg_x f, \deg_y f\}$.
\end{example}

We use the framework of \textit{motivic pieces of curves} set out in \cite[\S2]{DGKM}, the key parts of which we recall in \S\ref{sec:motiv}. For a curve $X$ admitting an action by automorphisms of the finite group $G$ and a $G$-representation $\rho$, define
    \[ X^\rho=\Hom_G(\rho,(V_\ell J_X)^*).\]
In \S\ref{sec:cyclic} we obtain relations between the conductors of these representations for curves with a cyclic action.

\begin{proposition}[=Proposition \ref{prop:cylic_general}]
    Let $K$ be a finite extension of $\QQ_p$. Given a $C_n$-cover $\pi:X\to B$, write $R_q$ for the number of points with ramification index divisible by $q$ for each $q\mid n$ prime. For $p>\max_{q\mid n}\{q,(R_q-2)(q-1)+1\}$ and $p>3$ if $2\mid n$, for any irreducible representation $\rho$ of $C_n$ we have
    \[ n_{\wild}(X^{\rho})=n_{\wild}(X^{\one}).\]
\end{proposition}

In \S\ref{sec:Sn}, using an Artin-like induction lemma, we piece together these cyclic relations, establishing Theorem \ref{thm:Sn_main} and hence Theorem \ref{thm:1}. Theorem \ref{thm:main} will follow from Theorem \ref{thm:Sn_main} by combining the following local constancy result with the local constancy of wild conductor exponents (e.g.\ \cite[Theorem 5.1(1)]{Kisin1999}).

\begin{lemma}[{=Lemma \ref{lem:Wconst}}]
    Let $K/\QQ_p$ be a finite extension. Let $g\in K[t]$ be $p^\text{th}$-power-free and suppose square-free $h\in K[t]$ is sufficiently close to $g$. We have $w_K(h)=w_K(g)$.
\end{lemma}

We have numerically tested Theorem \ref{thm:main} against known conductor exponents of curves. When testing against \cite[Proposition 4.1]{DokDoris}, which is used to compute wild conductor exponents at $p=2$ of genus 2 curves, we discovered a minor error in the latter. Therefore, lastly, in Appendix \ref{app:A} we fix this error.

\subsection*{Conventions} Throughout, a curve over $k$ is taken to be a geometrically connected, smooth, projective $k$-variety of dimension 1. We use the correspondence between finitely generated transcendence degree 1 extensions of a field $k$ and normal curves over $k$, and in particular---outside of \S\ref{sec:perturb}---we are referring to the unique normalisation of the projective closure whenever we describe a curve by a possibly-singular affine model. We take $\Proj^1$ to mean a genus $0$ curve with a rational point and hyperelliptic curves to be double covers of $\Proj^1$.

\stoptoc
\section*{\small{Acknowledgements}}
\resumetoc

My utmost gratitude goes to Vladimir Dokchitser for his wise supervision, without which this work would not have been possible. I thank Alexandros Konstantinou for his additional guidance and many fruitful discussions, Elvira Lupoian for helpful comments on a draft, James Rawson for helpful correspondence regarding Remark \ref{rmk:FultonVariant}, Tim Dokchitser for helpful correspondence regarding Appendix \ref{app:A}, and Andrew Obus and Padmavathi Srinivasan for sharing their work on Proposition \ref{prop:super}. Lastly, I thank the anonymous referee for their suggestions towards improving the clarity of the arguments herein.

This work was supported by the Engineering and Physical Sciences Research Council [EP/S021590/1], the EPSRC Centre for Doctoral Training in Geometry and Number Theory (The London School of Geometry and Number Theory), University College, London.

\section{Notation}

\vspace{0.2cm}

\renewcommand{\arraystretch}{1.5}
\begin{longtable}{p{.20\textwidth}  p{.80\textwidth}} 
        Notation & Terminology \\
        \hline
        $W_K\le G_K$ & Absolute Galois group and wild inertia subgroup \\
        $v_K$ & Normalised valuation on $p$-adic field $K$ \\
        $e_{K'/K}$, $f_{K'/K}$ & Ramification and inertia degrees of extension of $p$-adic fields \\
        $g_X$ & Genus of a curve $X$ \\
        $J_X$ & Jacobian variety associated to a curve $X$ \\
        $X/H$ & Quotient curve by the action of a finite group $H$ \\
        $T_\ell A$, $V_\ell A$ & $\ell$-adic Tate module of an abelian variety, $V_\ell A=T_\ell A\otimes_{\ZZ_\ell}\QQ_\ell$ \\
        $C_n$ & Cyclic group of order $n$ \\
        $S_n$ & Symmetric group on $n$ letters \\
        $F_n$ & Standard irreducible $S_n$-representation (cf.\ Definition \ref{def:standard}) \\
        $S_{n-1}^\circ$ & Subgroup stabilising a point under $S_n$-action on $n$ letters \\
        $\langle,\rangle$ & Representation-theoretic inner product on characters \\
        $\Ind_H^G \chi$ & Induction of a character $\chi$ of $H\le G$ to $G$ \\
        $\Res_H \chi$ & Restriction of a character $\chi$ of $G$ to $H\le G$ \\
        $\QQ(\chi)$ & Field generated by $\{\chi(g)\}_{g\in G}$ for a character $\chi$ of $G$
\end{longtable}

\section{Background}

\subsection{Wild conductor exponents}\label{sec:conductors}
We fix a finite extension $K$ of $\QQ_p$ and consider an $G_K$-representation $V$ over either $\QQ_l$ or $\FF_l$, where $\ell\ne p$.

\begin{definition}[{e.g.\ \cite[pages 3--4]{UlmerConductors}}]
    The \textit{wild conductor exponent} of $V$ is
    \[n_{\wild}(V)=\int_0^\infty \text{codim}(V^{G_K^u})\text{d}u,\]
    where the $G_K^u$ are higher ramification groups in upper numbering.
\end{definition}

In particular, $n_{\wild}(V)$ is determined only by the action of wild inertia. In fact, in the case of curves, the wild conductor exponent depends only on the action of wild inertia on $\ell$-torsion for any $\ell\ne p$.

Serre and Tate showed that the wild conductor exponents attached to abelian varieties vanish for $p$ sufficiently large relative to the dimension.

\begin{proposition}
    Let $A$ be an abelian variety of dimension $g$ over a finite extension of $\QQ_p$. If $p>2g+1$, then $n_{\wild}(V_\ell A)=0$ for all $\ell\ne p$.
\end{proposition}

\begin{proof}
    See the proof of \cite[Corollary 2]{SerreTate68}.
\end{proof}

We will make use of the following similar constraint.

\begin{lemma}\label{lem:0wild}
    Let $K$ be a finite extension of $\QQ_p$ and fix a prime $\ell\ne p$. Suppose there are abelian varieties $A_1,\hdots,A_n$ and $B_1,\hdots,B_m$ and some wild inertia representation $W$ of dimension $2d$ over $\FF_\ell$ such that
    \[W\oplus\bigoplus_i A_i[\ell]\cong \bigoplus_j B_j[\ell]\]
    as wild inertia representations. If $p>2d+1$, then $n_{\wild}(W)=0$.
\end{lemma}

\begin{proof}
    Suppose $p>2d+1$ and consider a wild inertia element $\sigma$ acting non-trivially on $\bigoplus_j B_j[\ell]$ as an element of order $p^r$. The characteristic polynomial of $\sigma$ has integer coefficients, e.g.\ by the well-known independence of $\ell$ of Weil--Deligne representations associated to abelian varieties.
    
    Therefore, each primitive root of unity appears as an eigenvalue of $\sigma$ on $W\oplus\bigoplus_i A_i[\ell]$ with equal multiplicity. The same is true for each $A_i[\ell]$, so each root of unity must appear as an eigenvalue of $\sigma$ acting on $W$. There must be at least $p-1>2d$ such eigenvectors, which is not possible, so the action on $W$ is trivial.
\end{proof}

Lastly, we note that one can keep track of wild conductor exponents under tamely ramified base change.

\begin{lemma}\label{lem:TameExtension}
    Let $V$ be an $\ell$-adic representation over a finite extension $K$ of $\QQ_p$. Suppose $K'/K$ is a tamely ramified extension. We have
    \[ n_{\wild}(\Res_{G_{K'}}V) = e_{K'/K}\cdot n_{\wild}(V). \]
\end{lemma}

\begin{proof}
    This follows from the fact that $G_K^u \cap G_{K'}=G_{K'}^{ue_{K'/K}}$.
\end{proof}

\subsection{Galois covers and Galois closure}
Recall that a cover of curves is a surjective morphism $\pi: X\to B$ of curves over a field $K$. Functorially, we obtain an embedding of function fields $K(B)\hookrightarrow K(X)$.

\begin{definition}
    We say $\pi$ is a \textit{Galois cover} if $K(X)/K(B)$ is a Galois extension, and we will say $\pi$ is a $G$-cover if $\pi$ is Galois, $\Gal(K(X)/K(B))\cong G$ and $K(X)$ contains no algebraic extension of $K$.

    Given a non-Galois cover $\pi: X\to B$, we say the \textit{Galois closure} is the curve whose function field is the Galois closure of the extension $K(X)/K(B)$. We say that a cover has $G$-Galois closure if its Galois closure is a $G$-cover.
\end{definition}

We will often write `cyclic cover', `symmetric cover' etc.\ to mean a Galois cover of curves whose Galois group has this property.

We will be particularly interested in the ramification properties of covers and especially in the case where covers are simply branched.

\begin{definition}\label{def:simp}
    A degree $n$ cover of curves $\pi:X\to B$ is \textit{simply branched} if for each $P\in B$ we have $|\pi^{-1}(P)|\ge n-1$. 
    
    Alternatively, consider the Galois closure $\tilde{\pi}:\tilde{X}\to B$ of $\pi: X\to B$, writing $G=\Gal(\pi)$ and take $H\le G$ such that $X=\tilde{X}/H$. Then $\pi$ is simply branched if the decomposition group of $\PP\in \tilde{X}$ above $P\in B$ acts either trivially or by a single transposition on $G/H$.
\end{definition}

\subsection{Kernels of pull-backs and push-forwards}
Given a cover of curves $\pi: X\to B$, we define the Prym variety associated to $\pi$ to be the connected component of $\ker(\pi_*: J_X\to J_B)$ containing the identity, denoted $\Prym(\pi)$. The following well-known proposition governs the index of $\Prym(\pi)$ in $\ker\pi_*$ and, dually, the kernel of the pull-back $\pi^*:J_B\to J_X$.

\begin{proposition}\label{prop:push-pull}
    Let $p$ be a prime and $\pi: X\to B$ a $C_p$-cover of curves over a field $K$ of characteristic different from $p$. Either
    \begin{enumerate}
        \item $\pi$ is unramified, $\ker(\pi_*)/\Prym(\pi)\cong \langle T\rangle$ for some $T\in J_X[p]$ and $\ker\pi^*=\langle P\rangle$ for some $P\in J_B[p]$, or
        \item $\pi$ is ramified, $\ker\pi_*/\Prym(\pi)=1$ and $\pi^*$ is injective.
    \end{enumerate}
\end{proposition}

\begin{proof}
    For (1): see the main theorem of \cite{RosenNorms} for the claim about $\ker\pi_*$ and \cite[Lemma 2.2]{AgostiniPryms} for the claim about $\ker\pi^*$. 

    For (2): we will show that $\pi^*$ is injective; the claim about $\ker\pi_*$ is equivalent by Cartier duality, e.g.\ as shown in the proof of \cite[Proposition 3.3]{ConradStein01}.

    Because $\pi_*\circ\pi^*=[p]$, we have $\ker\pi^*\subseteq J_B[p]$. Now, subgroups generated by points $0\ne P\in J_B[p]$ correspond by Kummer theory to unramified $C_p$-covers $B_P\to B$ over $\bar{K}$, or equivalently to unramified $C_p$-extensions $\bar{K}(B_P)/\bar{K}(B)$. In terms of function fields, the pull-back $\pi^*(P)$ corresponds to the extension $\bar{K}(X)\bar{K}(B_P)/\bar{K}(X)$, which is an unramified $C_p$-extension of $\bar{K}(X)$. The existence of $P\in\ker\pi^*$ would say $\bar{K}(X)\bar{K}(B_P)=\bar{K}(X)$, a contradiction.
\end{proof}

\subsection{Motivic pieces of curves}\label{sec:motiv}

We summarise the key definition and some basic properties from \cite[\S 2]{DGKM}. Fix a $G$-cover of curves $\pi: X\to B$ over a finite extension $K$ of $\QQ_p$.

Consider the action (inherited from that on the points of $X$) of $G$ on $V_\ell J_X$, noting that this commutes with the action of $G_K$ because our cover is $K$-rational.

\begin{definition}[{=\cite[Definition 2.3]{DGKM}}]
    For a $G$-representation $\rho$, define
    \[ X^\rho=\Hom_G(\rho,(V_\ell J_X)^*),\]
    on which $G_K$ acts by postcomposition. Implicitly this requires a choice of $\ell$, but the usual independence properties hold (cf.\ \cite[Corollary 2.14]{DGKM}).
\end{definition}

Recall that a virtual character is a $\ZZ$-linear combination of characters.

\begin{definition}
    For a virtual character $\chi = \sum_i r_i\rho_i$ of $G$, we define
    \[n_{\wild}(X^\chi)=\sum_i r_i\cdot n_{\wild}(X^{\rho_i}).\]
\end{definition}

Note that Galois-conjugate pieces will have the same wild conductor exponents:

\begin{lemma}\label{lem:conj_cond}
    Let $X$, $G$ and $\rho$ be as above. For $\sigma \in \Gal(\QQ(\rho)/\QQ)$, we have
    \[ n_{\wild}(X^\rho) = n_{\wild}(X^{\sigma(\rho)}). \]
\end{lemma}

\begin{proof}
    Choosing $\ell$ prime to $p$, \cite[Lemma 2.12]{DGKM} implies that 
    \[\Tr(\alpha | H^1_\ell(X^\rho))^\sigma = \Tr(\alpha | H^1_\ell(X^{\sigma(\rho)})) \]
    for $\alpha\in W_K$. Wild inertia acts through a finite quotient, say $W$, so we have that $H^1_\ell(X^\rho)$ and $H^1_\ell(X^{\sigma(\rho)})$ are conjugate $W$-representations. That the wild conductor exponents are equal follows.
\end{proof}

The following key properties come from an analogue of the `Artin formalism'.

\begin{lemma}\label{lem:induction}
    Fix ${X}$ and $G$ as above. Let $H$ and $N$ be a subgroup and a normal subgroup of $G$, respectively.
    \begin{enumerate}
        \item For $\rho_1$, $\rho_2$ $G$-representations: $n_{\wild}({X}^{\rho_1+\rho_2})=n_{\wild}({X}^{\rho_1})+n_{\wild}({X}^{\rho_2}).$
        \item For $\rho$ an $H$-representation: $n_{\wild}({X}^{\Ind_H^G \rho})=n_{\wild}({X}^\rho)$.
        \item For $\rho$ a lift of $\rho'$, a $G/N$-representation: $n_{\wild}(X^\rho)=n_{\wild}((X/N)^{\rho'}).$
    \end{enumerate}
\end{lemma}

\begin{proof}
    (1) and (2) follow immediately from \cite[Proposition 2.8]{DGKM}. For (3), combine with [\textit{loc.\ cit.}, Remark 2.6].
\end{proof}

A special case of Lemma \ref{lem:induction} is the following observation, which shows how we translate between wild conductor exponents of curves and of motivic pieces.

\begin{lemma}\label{lem:curvinduct}
    For $H\le G$, we have
    \[n_{X/H,\wild}=n_{\wild}(X^{\Ind_H^G \one}).\]
\end{lemma}

\begin{notation}
    We write $S_{n-1}^\circ$ for the subgroup of $S_n$ which stabilises $n$.
\end{notation}

\begin{example}
    Consider an $S_n$-cover $X\to B$. Let $F$ be the standard irreducible representation of $S_n$ (cf.\ Definition \ref{def:standard}). We have 
    \[n_{X/S_{n-1}^\circ,\wild}=n_{\wild}(X^{F+\one}).\qedhere\]
\end{example}

\section{Cyclic covers}\label{sec:cyclic}

In this section, given a cyclic cover of curves $X\to B$, we establish relations between wild conductor exponents of $X$ and $B$. We first treat the case of cyclic covers of prime order, before moving onto general cyclic groups.

\subsection{Cyclic groups of prime order}\label{sec:cyclic_prime}

The first case we treat is that of $C_2$-covers of curves.

\begin{proposition}\label{prop:unramC2}
    Suppose $\pi:X\to B$ is an $C_2$-cover of curves over a finite extension of $\QQ_p$, ramified at $R$ points. For $p>\max\{3,R-1\}$, we have
    \[n_{X,\wild}=2\cdot n_{B,\wild}.\]
    Equivalently, for $\eps$ the non-trivial irreducible representation of $C_2$, we have
    \[ n_{\wild}(X^\eps)=n_{\wild}(X^{\one}). \]
\end{proposition}

\begin{proof}
    In the unramified case, Riemann--Hurwitz gives $g_X=2g_B-1$. Also in this case, the map $\pi_*:J_X[2]\to J_B[2]$ has kernel of dimension $2g_X-2g_B+1=2g_B-1$ by Proposition \ref{prop:push-pull} and the fact that $\dim\Prym(\pi)=2g_X-2g_B$. The image $\pi_*(J_X[2])$ then has image of dimension $2g_X-(2g_X-2g_B+1)=2g_B-1$ and, by Proposition \ref{prop:push-pull}, $\pi^*$ has kernel of dimension $1$.
    
    We can apply Maschke's Theorem (e.g.\ \cite[Theorem 1]{SerreReps}) to $\pi_*(J_X[2])\le J_B[2]$ and $\ker\pi_*[2]\le J_X[2]$ because wild inertia acts through a quotient of $p$-power order, and so by a group of order prime to $2$. Combining with the first isomorphism theorem, this gives a one-dimensional wild inertia representation $V$ such that $J_X[2]\oplus V\cong J_B[2]\oplus \ker\pi_*[2]$.

    Also, because $\pi_*\circ\pi^*=[2]$, we have $\pi^*J_B[2]\le \ker\pi_*[2]$ and counting gives an equality. Applying Maschke's Theorem to $\ker\pi^*\le J_B[2]$, there is another one-dimensional wild inertia representation $V'\cong\ker\pi^*$ such that $J_B[2]\cong\pi^*J_B[2]\oplus V'$, so $J_X[2]\oplus V'\oplus V \cong J_B[2]\oplus J_B[2]$. We have $n_{\wild}(V\oplus V')=0$ by Lemma \ref{lem:0wild}, using that $p>3$.

    In the ramified case, Proposition \ref{prop:push-pull} gives that $\pi_*$ is surjective on $2$-torsion and $\pi^*$ is injective so it is even more straightforward: we have in the same way, using Riemann--Hurwitz, an $(R/2-1)$-dimensional $V$ such that $J_X[2]\oplus V \cong J_B[2]\oplus J_B[2]$ and again $n_{\wild}(V)=0$ by Lemma \ref{lem:0wild}, using that $p>R-1$.
    
    The second claim is equivalent because $n_{B,\wild}=n_{\wild}(X^{\one})$ and $n_{X,\wild}=n_{\wild}(X^{\one+\eps})$ by Lemma \ref{lem:curvinduct}.
\end{proof}

We now treat the case of cyclic covers $C_q$ for odd primes $q$, writing $\tau$ for a generator of $C_q$.
 
\begin{proposition}\label{prop:Cq_main}
    Consider a $C_q$ cover of curves $\pi: X\to B$ over a finite extension of $\QQ_p$. Write $R_q$ for the number of ramification points. If $p>\max\{q,(R_q-2)(q-1)+1\}$, then
    \[n_{X,\wild}=q\cdot n_{B,\wild}.\]
    Equivalently, for $\rho$ a non-trivial irreducible representation of $C_q$, we have
    \[ n_{\wild}(X^\rho)=n_{\wild}(X^{\one}). \]
\end{proposition}

\begin{proof}
    By Riemann--Hurwitz, we have $g_X=q\cdot g_B + (R_q-2)(q-1)/2$.

    Now, as an endomorphism on $\Prym(\pi)$, the map $(1-\tau)^{q-1}$ differs from $[q]$ by a unit. This can be seen because $\ZZ[\zeta_q]$ embeds into $\End(\Prym(\pi))$ by sending $\zeta_q$ to $\tau$. Therefore, $\Prym(\pi)[1-\tau]$, the kernel of $1-\tau$ on the Prym, has $\FF_q$-dimension $2(g_X-g_B)/(q-1)=2g_B+R_q-2$ and so $\Prym(\pi)[(1-\tau)^{j}]$ has dimension $j \cdot (2g_B+R_q-2)$.

    Note that, for $R_q\ge2$, we have $\pi^*J_B[q]\le\Prym(\pi)[1-\tau]$ and so 
    \[\Prym(\pi)[1-\tau]\cong V\oplus \ker((1-\tau)\circ\pi^*)[q]\cong V\oplus J_B[q]\] 
    for some $(R_q-2)$-dimensional wild inertia representation $V$ by Maschke's Theorem. In this case, for each $2\le i\le q-1$, we have
    \[ \Prym(\pi)[(1-\tau)^i]/\Prym(\pi)[(1-\tau)^{i-1}] \cong V\oplus J_B[q].\]
    To see this, note that
    \[\Prym(\pi)[(1-\tau)^i]\xrightarrow{(1-\tau)^{i-1}} (1-\tau)^{i-1}(\Prym(\pi))\cap \Prym(\pi)[1-\tau],\]
    is surjective with kernel $\Prym(\pi)[(1-\tau)^{i-1}]$, and so the right-hand side is simply $\Prym(\pi)[1-\tau]$ by counting.

    We decompose $J_X[q]$ as a wild inertia representation:
    \begin{equation}
        J_X[q] \cong J_X[q]/\ker\pi_*[q] \oplus \bigoplus_{i=1}^{q-1} \Prym(\pi)[(1-\tau)^i]/\Prym(\pi)[(1-\tau)^{i-1}].
    \end{equation}
    Using that $J_X[q]/\ker\pi_*[q]\cong\pi_*(J_X[q])$ and that $\pi_*(J_X[q])=J_B[q]$, we have
    \begin{equation}\label{eqn:cond_sum}\tag{$\dagger$}
        J_X[q]\cong J_B[q]^{\oplus q}\oplus V^{\oplus q-1}
    \end{equation}
    as wild inertia representations. Note that $\pi_*(J_X[q])=J_B[q]$ holds because $\ker\pi_*=\Prym(\pi)$ by Proposition \ref{prop:push-pull}, which says that $\ker\pi_*[q]$ has dimension $2g_X-2g_B$, so $\pi_*(J_X[q])$ has dimension $2g_X - (2g_X-2g_B) = 2g_B$. We reach the first conclusion by applying Lemma \ref{lem:0wild} to $V^{\oplus q-1}$, using that $p>\dim V^{\oplus q-1} +1 = (R_q-2)(q-1)+1$. 

    In the case $R_q=0$, there is an analogous identity to \eqref{eqn:cond_sum} with the `junk' terms instead on the left-hand side. One can see this by using the facts that $\pi^*(J_B[q])$ is all of $\ker\pi_*[q]$ and that $\pi_*: J_X[q]\to J_B[q]$ has cokernel of dimension one, then using that $p>q$.

    A monodromy argument shows that we cannot have $R_q=1$. An alternative way to see this is that, assuming $R_q=1$, we would have $\pi^*J_B[q]\le \Prym(\pi)[1-\tau]$, but the left-hand side has dimension $2g_B$ whilst the right has dimension $2g_B-1$. 

    For the second part of the proposition, apply Lemma \ref{lem:curvinduct} to write 
    \[n_{X,\wild}=n_{\wild}(X^{\one})+\sum_{\sigma\in\Gal(\QQ(\zeta_q)/\QQ)}n_{\wild}(X^{\sigma(\rho)}) \text{ and } n_{B,\wild}= n_{\wild}(X^{\one})\]
    and conclude by Lemma \ref{lem:conj_cond}, which says that these conjugate pieces have the same wild conductor exponents.
\end{proof}

\begin{remark}\label{rmk:super}
    In the case of superelliptic curves $X: y^n=f(x)$, we have $J_X[1-\tau]$ being generated by differences of ramification points, which are easily described: they correspond to Weierstrass points $(\alpha,0)$ with ramification index $n/\gcd(n,m_\alpha)$, where $m_\alpha$ is the multiplicity of $\alpha$ as a root of $f$. One can show that, for a cover $X\to B$ with $B: y^{n/q}=f(x)$, the wild inertia action on $V$ is given by the actions of wild inertia on the roots of the irreducible factors of $f$ whose multiplicity $d$ satisfies $\gcd(n,d)=\gcd(n/q,d)$. The corresponding conductor exponents are measured by $w_K$. This sets up a proof of Proposition \ref{prop:super} by induction: use this description of $V$ and take wild conductor exponents on both sides of \eqref{eqn:cond_sum}.
\end{remark}

\subsection{General cyclic groups}
We now consider a general $C_n$-cover of curves $X\to B$ over a finite extension of $\QQ_p$.

\begin{proposition}\label{prop:cylic_general}
    Consider a $C_n$-cover of curves $\pi:X\to B$. For each prime $q\mid n$, write $R_q$ for the number of points with ramification index divisible by $q$ in this cover. For $p>\max_{q\mid n}\{q,(R_q-2)(q-1)+1\}$ and $p>3$ if $2\mid n$, we have
    \[n_{\wild}(X^\rho)=n_{\wild}(X^{\one})\]
    for any irreducible $C_n$-representation $\rho$. 
\end{proposition}

\begin{proof}
    We go by induction on $n$.
    
    Choose prime $q\mid n$ and consider $C_q\le C_n$. Writing $\rho'$ for a non-trivial irreducible $C_q$-representation, we have
    \[n_{\wild}(X^{\Ind_{C_q}^{C_n}(\rho'-\one)})=0\]
    by one of Propositions \ref{prop:unramC2} or \ref{prop:Cq_main} and Lemma \ref{lem:induction}(2), the former of which apply because the maximum number of points with ramification index $q$ in the $C_q$-subcover is $R_q$.

    We may decompose $\Ind_{C_q}^{C_n}(\rho'-\one)$ as some degree zero $\ZZ$-linear combination of irreducibles $\rho_i$. For non-faithful $\rho_i$, we have $n_{\wild}(X^{\rho_i})=n_{\wild}(X^{\one})$ by passing to the quotient on which $\rho_i$ is faithful, by the inductive hypothesis and Lemma \ref{lem:induction}(3). Note that we can pass to quotients because the quotient cover will also have at most $R_q$ points with ramification index divisible by $q$.

    By Frobenius reciprocity, all faithful irreducibles must appear amongst the $\rho_i$. Combining with the above gives $\sum_{\rho\text{ faithful}}n_{\wild}(X^{\rho}) =\phi(n)\cdot n_{\wild}(X^{\one})$, where $\phi$ is Euler's totient function and so $\phi(n)$ is the number of faithful irreducible representations. The result follows from Lemma \ref{lem:conj_cond} because all the faithful irreducibles are conjugate.
\end{proof}

\section{Symmetric covers and Theorem \ref{thm:1}}\label{sec:Sn}
The aim of this section is to establish Theorem \ref{thm:1} by piecing together relations coming from cyclic subgroups of $S_n$.

\begin{definition}\label{def:standard}
    For a field $K$ of characteristic different from $n$, let $S_n$ act on $K^n$ by permuting the standard basis vectors. The standard irreducible representation $F=F_n$ of $S_n$ is the $(n-1)$-dimensional subspace consisting of vectors whose coefficients sum to $0$.
\end{definition}

Write $\eps$ for the sign representation and, as always, $\one$ for the trivial irreducible representation of $S_n$. We seek to prove:

\begin{theorem}\label{thm:Sn_main}
    Let $X\to B$ be an $S_n$-cover of curves over a finite extension of $\QQ_p$ with $p>n$, such that $X/S_{n-1}^\circ\to B$ is simply branched, except possibly above one branch point. We have
    \[n_{\wild}(X^{F})=(n-2)\cdot n_{\wild}(X^{\one})+n_{\wild}(X^\eps).\]
\end{theorem}

Once we have proved Theorem \ref{thm:Sn_main}, Theorem \ref{thm:1} will easily follow:

\begin{proof}[Proof of Theorem \ref{thm:1}]
    By \cite[Corollary 3.2]{BKP2024}, a simply branched cover $C\to\Proj^1$ of degree $n$ has $S_n$-Galois closure, say $X$. We take $D=X/A_n$. After using Lemma \ref{lem:curvinduct} to write $n_{C,\wild}=n_{\wild}(X^{\one+F})$ and $n_{D,\wild}=n_{\wild}(X^{\one\oplus\eps})$---and noting that $n_{\wild}(X^{\one})=n_{\Proj^1,\wild}=0$---we have $n_{C,\wild}=n_{D,\wild}$ from the simply branched case of Theorem \ref{thm:Sn_main}. Finally, taking another hyperelliptic curve with the same branch locus just replaces $D$ by a quadratic twist $D'$, whose Jacobian has isomorphic $2$-torsion to that of $D$ and so the same wild conductor exponent. 
\end{proof}


\subsection{The case $n=3$}

The case of $S_3$-covers is particularly interesting for two reasons; because in some instances we can extract more information than about just the wild inertia action, and because it affords generalisation to dihedral groups.

Consider an $S_3$-cover of curves $\pi: X \to B$ over a finite extension of $\QQ_p$, labelling the quotients as in the diagram below.

\[\begin{tikzcd}
	& X \\
	{C=X/C_2} \\
	&& {D=X/C_3} \\
	& {B=X/S_3}
	\arrow["{\pi_{X,C}}"', from=1-2, to=2-1]
	\arrow["{\pi_{X,D}}", from=1-2, to=3-3]
	\arrow["{\pi_{X,B}}", from=1-2, to=4-2]
	\arrow["{\pi_{C,B}}"', from=2-1, to=4-2]
	\arrow["{\pi_{D,B}}", from=3-3, to=4-2]
\end{tikzcd}\]

Let $\rho$ be a non-trivial representation of $C_3$. By Proposition \ref{prop:Cq_main}, we have $n_{\wild}(X^\rho)=n_{\wild}(X^{\one_{C_3}})$ for $p>\max\{3,2(R_3-2)+1\}$, where $R_3$ is the number of ramification points in the subcover $X\to D$. Applying Lemma \ref{lem:induction}(2), we obtain a relation 
\[n_{\wild}(X^{\Ind_{C_3}^{S_3}(\rho-\one_{C_3})})=n_{\wild}(X^{F_3-\one-\eps})=0.\] 
After noting that $C\to B$ being non-simply branched above a single point corresponds to $R_3=2$ (e.g.\ by Riemann--Hurwitz), this is precisely the case $n=3$ of Theorem \ref{thm:Sn_main}. In particular, we find
\[ J_C[3] \sim J_B[3] \oplus J_D[3] \]
as wild inertia representations, where $\sim$ denotes isomorphism up to trivials. In this case we can, however, go beyond a statement about wild inertia.

\begin{proposition}\label{prop:S3_case}
    Let $X\to\Proj^1$ be an $S_3$-cover of curves and let $R$ be the number of points of ramification index $3$ in the quotient cover $C\to\Proj^1$. For any choice of $\langle\tau\rangle \cong C_3\le S_3$, the $G_K$-module map 
    \begin{equation}\tag{$\ddagger$}\label{eqn:GalIsom}
        (1-\tau)\circ \pi_{X,C}^* \colon J_C[3] \to J_X[3]
\    \end{equation}
    has image landing in $\pi_{X,D}^*J_D[3]$, into/onto which it is
    \begin{enumerate}
        \item an injection if $R=0$ (equivalently, if $g_C=g_D-1$);
        \item an isomorphism if $R=1$ (equivalently, if $g_C=g_D$).
    \end{enumerate}
\end{proposition}

\begin{proof}
    We need to show that the image lands in $\pi_{X,D}^*J_D[3]$. Firstly, note that a point $P$ in the image satisfies $(1-\tau)(P)=0$. This is because 
    \[(1-\tau)^2 = 1+\tau+\tau^2-3\tau \equiv 1+\tau+\tau^2\pmod{3},\]
    and $1+\tau+\tau^2=0$ on $\pi^*_{X,C}(J_C)$. Indeed, for $P\in J_C$, the divisor $(1+\tau+\tau^2)\circ\pi_{X,C}^*(P)$ on $X$ is the pull-back of $\pi_{C,\Proj^1,*}P$ on $\Proj^1$. Therefore, the image lies in $\ker(1-\tau)|_{\ker(\pi_{X,D,*})[3]}$, which contains $\pi_{X,D}^*J_D[3]$. To show equality between these two spaces, we split into cases:

    In the $R=1$ case, a Riemann--Hurwitz calculation shows that $g_C=g_D$ and that $\pi_{X,D}$ is branched over exactly two points. As in the proof of Proposition \ref{prop:Cq_main}, we have $\ker(1-\tau)|_{\ker(\pi_{X,D,*})[3]}=\Prym(\pi_{X,D})[1-\tau]$ having dimension $2g_D$ and so this space is $\pi_{X,D}^*J_D[3]$, using that $\pi^*_{X,D}$ is injective by Proposition \ref{prop:push-pull}.

    In the $R=0$ case, we have $g_C=g_D-1$ and $\pi_{X,D}$ being unramified. Here we have $\ker(1-\tau)|_{\ker(\pi_{X,D,*})[3]}$ of dimension at most one more than $\Prym(\pi_{X,D})[1-\tau]$ by Proposition \ref{prop:push-pull}, which has dimension $2g_D-2$ as in the proof of Proposition \ref{prop:Cq_main}. Moreover, the former contains $\pi_{X,D}^*J_D[3]$ of dimension $2g_D-1$, so we must again have $\ker(1-\tau)|_{\ker(\pi_{X,D,*})[3]}=\pi_{X,D}^*J_D[3]$.

    We now think of the codomain of \eqref{eqn:GalIsom} as being $\pi_{X,D}^*J_D[3]$: in the $R=1$ case, the space of points $P\in J_C[3]$ such that $\pi_{X,C}^*P$ is fixed by $\tau$ is trivial because $\pi_{X,D}^*J_D[3]$ intersects trivially with $\pi_{X,C}^*J_C[3]$. To see this, note that $\pi_{X,C}^*P=\pi_{X,D}^*Q$ implies $2P=0$ by applying $\pi_{X,C,*}$ to both sides. It immediately follows that the map must be an isomorphism.
    
    The claim in the unramified case follows from the same observation, which also shows that no points in $\pi_{X,C}^*J_C[3]$ can be fixed by $\tau$.
\end{proof}

\begin{remark}
    One could also state a version of Proposition \ref{prop:S3_case} for $R\ge2$, in which case \eqref{eqn:GalIsom} would have image $\pi_{X,D}^*J_D[3]$, but the details of such a proof grow increasingly cumbersome.
\end{remark}

\begin{remark}
    In the case of an elliptic curve $C: y^2=x^3+ax+b$, $a\ne0$, equipped with the cover $x: C\to \Proj^1$, we recover the known isomorphism of Galois modules between $C[3]$ and $J_D[3]$. To the best of the author's knowledge, this observation first appeared in the literature as \cite[Lemma 6.9 + Remark 6.10]{DGKM}, where they study the kernel of an isogeny $J_C^2\times J_D \to J_X$.
\end{remark}

\begin{remark}
    For an odd prime $q$ and a $D_{2q}$-cover $X\to B$, where $D_{2q}$ is the dihedral group of order $2q$, with quotients labelled analogously to the above, one can similarly show
    \[ J_C[q] \sim J_B[q] \oplus J_D[q]^{\oplus (q-1)/2} \]
    using the techniques of Proposition \ref{prop:Cq_main}. Moreover, when $R_q=1$, we have equality on wild conductor exponents between $J_C$ and $J_B\times J_D^{(q-1)/2}$ for all primes different from $q$. This is an alternative generalisation of the case $n=3$ of Theorem \ref{thm:Sn_main}.
\end{remark}

\subsection{The general case}

We prove Theorem \ref{thm:Sn_main} in the general case by inducing relations from cyclic subgroups. First we prove some required Galois and representation-theoretic lemmata.

\begin{lemma}\label{lem:ram}
    Let $X\to B$ be an $S_n$-cover such that $\pi:X/S_{n-1}^\circ\to B$ is simply branched, except possibly above one branch point. If $C_m\le S_n$ contains no transposition, then for any $q\mid m$, the cover $X\to X/C_m$ has at most $\lfloor n/q \rfloor$ points with ramification index divisible by $q$. Moreover, when $\pi$ is simply branched, this cover is unramified.
\end{lemma}

\begin{proof}
    Fix a point $P\in B$, and choose a point $\PP$ on $X$ lying above $P$. The points above $P$ on $X/S_{n-1}^\circ$ correspond to the orbits of $\{1,\hdots,n\}$ under the action of the decomposition group at $\PP$. The assumption that $X/S_{n-1}^\circ\to B$ is simply branched away from a particular branch point thereby implies that, as we vary over $P$, each decomposition group is either trivial or generated by a transposition, except for those above this branch point.

    The result follows because all decomposition groups of $X\to X/{C_m}$ are trivial, except possibly those of the points above this branch point.
\end{proof}

The following is essentially a version of Artin's Induction Theorem (cf.\ \cite[\S9.2]{SerreReps}).

\begin{lemma}\label{lem:SnReps}
    Given a virtual character $\chi$ of $S_n$ with degree $0$, there exists non-zero integers $a_i$, not-necessarily-distinct cyclic subgroups $H_i$ and virtual characters $\Theta_i$ of the form
        \[\Theta_i=\rho_i-\one_{H_i},\] 
    where $\rho_i$ is an irreducible $H_i$-representation, such that
    \[a_0\cdot \chi=\sum_i a_i\cdot \Ind_{H_i}^{S_n}\Theta_i.\]
\end{lemma}

\begin{proof}
    Suppose $\langle\chi, \Ind_H^{S_n}\Theta\rangle=0$ for all cyclic subgroups $H$ and possible characters $\Theta$ of this form. It suffices to show that $\chi=0$.

    Given a cycle type, choose a permutation of this type and let $H$ be the subgroup generated by this permutation. We have $\langle\Res_H \chi,\Theta\rangle=0$ for all characters $\Theta$ of $H$ of the above form, by assumption. Therefore, $\Res_H\chi=0$ because these $\Theta$ obviously generate the set of characters of $H$ with degree $0$. Thus $\chi$ vanishes on every conjugacy class.
\end{proof}

Before proving Theorem \ref{thm:Sn_main}, we single out the special case of $C_2\le S_n$ generated by a transposition.

\begin{lemma}\label{lem:C2res}
    Let $C_2\le S_n$ be generated by a transposition. We have 
    \[\Res_{C_2} (F_n-\eps-(n-2)\cdot\one) = 0.\]
\end{lemma}

\begin{proof}
    First note that, when $n=3$, we have $\Res_{C_2}(F_3)=\one_{C_2} \oplus \eps_{C_2}$, where $\eps_{C_2}$ is the non-trivial irreducible of $C_2$.

    For general $n$, restricting to $S_{n-1}^\circ$ gives $\Res_{S_{n-1}^\circ}F_n=F_{n-1} + \one_{S_{n-1}^\circ}$ by Frobenius reciprocity, and the result follows by induction.
\end{proof}

We are now ready to prove Theorem \ref{thm:Sn_main}.

\begin{proof}[Proof of Theorem \ref{thm:Sn_main}]
    We treat the simply branched case:
    
    Note that $F-\eps-(n-2)\cdot\one$ has degree $0$, so we may find $a_i$, $H_i$ and $\Theta_i$ as in Lemma \ref{lem:SnReps} such that 
    \begin{equation}\label{eqn:induct}\tag{$*$}
        a_0\cdot (F-\eps-(n-2)\cdot\one)=\sum_i a_i\cdot \Ind_{H_i}^{S_n}\Theta_i.
    \end{equation}
    If $H_i$ contains no transposition, then $n_{\wild}(X^{\Ind_{H_i}^{S_n}\Theta_i})=0$ by Lemma \ref{lem:induction}(2) and Proposition \ref{prop:cylic_general}, which applies by Lemma \ref{lem:ram}. Our aim is to show that we can take all $\Theta_i$ appearing in \eqref{eqn:induct} to also satisfy $n_{\wild}(X^{\Ind_{H_i}^{S_n}\Theta_i})=0$, which will then give $n_{\wild}(X^{F-\eps-(n-2)\cdot\one})=0$.
    
    The cycle types which generate subgroups containing transpositions are transpositions themselves along with the product of cycles of odd lengths with a transposition. By Frobenius reciprocity,
    \[\langle F-\eps-(n-2)\cdot\one,\Ind_{H}^{S_n}\Theta\rangle=\langle \Res_{H}( F-\eps-(n-2)\cdot\one, \Theta \rangle\]
    for any subgroup $H\le S_n$ and character $\Theta$ thereof. For $H$ generated by a single transposition, Lemma \ref{lem:C2res} says that this restriction is the zero character, so we may assume that no such $H$ appears in \eqref{eqn:induct}.

     For $H$ generated by the product of a transposition with cycles of odd length, identify $H$ with $C_{2m}\cong C_2\times C_m$. We will show that $\Res_H(F-\eps-(n-2)\cdot\one)$ is the lift of a character from $C_{2m}/(C_2\times\{1\})$, which we identify with $\{1\}\times C_m$:
    
    Lemma \ref{lem:C2res} applied to $C_2\times\{1\}\le C_2\times C_m$ gives $\Res_{C_2\times\{1\}}F=(n-2)\cdot\one_{C_2\times\{1\}}+\eps_{C_2\times\{1\}}$, so $\Res_{C_2\times C_m} F$ is a sum of $(n-2)$ irreducibles lifted from $\{1\} \times C_m$ and one irreducible representation which restricts to $\eps_{C_2\times\{1\}}$ on $C_2\times\{1\}$, which must be $\Res_{C_2\times C_m}\eps$ because $F$ is rational.

    We may now assume that, for any $H_i$ of this form appearing in \eqref{eqn:induct}, $\Theta_i$ is the lift of a character $\Theta_i'$ of $\{1\}\times C_m$. We then have
    \[
\renewcommand{\arraystretch}{1.5}
\begin{array}{r@{\;}c@{\;}l@{\qquad}l}
    n_{\wild}(X^{\Ind_{H_i}^{S_n}\Theta_i}) &=& n_{\wild}(X^{\Theta_i})
    & \annot{Lemma \ref{lem:induction}(2)} \\
  &=& n_{\wild}((X/(C_2\times\{1\}))^{\Theta_i'})
    & \annot{Lemma \ref{lem:induction}(3)} \\
    & = & 0
    & \annot{Proposition \ref{prop:cylic_general}}
\end{array}
\]
    and so conclude. Note that Proposition \ref{prop:cylic_general} applies because $X/(C_2\times \{1\})\to X/(C_2\times C_m)$ is unramified by Lemma \ref{lem:ram}, e.g.\ because $m$ is odd.

    The case in which there is a non-simply branched point is almost identical, but $C_m\le S_n$ may have up to $\lfloor n/q\rfloor$ ramification points with ramification index divisible by $q$ by Lemma \ref{lem:ram}. Plugging this into the bounds from Proposition \ref{prop:cylic_general} shows that we may use the proposition in the same way as in the unramified case because we are assuming $p>n$.
\end{proof}

\begin{remark}
    In the case $n=3$, we saw that we have a $G_K$-module isomorphism \eqref{eqn:GalIsom} when $B=\Proj^1$ and there is precisely one point with ramification index $3$. In this case we also have equality, for example, between the wild conductor exponents at $2$. 

    In this sense, Theorem \ref{thm:Sn_main} is the best one could hope for when $n>3$ because we must use relations on $q$-torsion for all $q<n$ prime, but this process obscures the conductor exponent at each such $q$.
\end{remark}

\begin{remark}\label{rmk:FultonVariant}
    It is a result credited by Fulton to Severi (\cite[Proposition 8.1]{Fulton69}) that a curve $C$ over an algebraically closed field admits a simply branched degree $g_C+1$ cover of $\Proj^1$. In the case of curves over finite extensions of $\QQ_p$, we can ensure that such covers are defined over finite tame extensions. Thereby, one could use Theorem \ref{thm:Sn_main} and Lemma \ref{lem:TameExtension} to compute wild conductor exponents for $p>g_C+1$. In practice, given a particular curve, a perturbation argument similar to that in the sequel would likely work for smaller $p$.
\end{remark}

\section{Perturbations and Theorem \ref{thm:main}}\label{sec:perturb}

In this section we institute a change of perspective: instead of considering a Galois cover $X\to B$, we consider a non-Galois cover of degree $n$ from $C\to B$ and its Galois-closure $X$. Further, we restrict to the case that $B=\Proj^1$. Note that (at least generically) $C$ plays the role of $X/S_{n-1}^\circ$ as in the preceding sections.

Recall the following piece of notation.

\begin{notation}
    For $K$ a finite extension of $\QQ_p$ and a polynomial $g\in K[t]$ write
    \[w_K(g)=\sum_{r\in R/G_{K}} m(r)\cdot(v_{K}(\Delta_{K(r)/K})-[K(r):K]+f_{K(r)/K}), \]
    where $R$ is the set of roots of $g$ over $\bar{K}$ and $m(r)$ is the multiplicity of $r$.
\end{notation}

This quantity will arise because of its connection to wild conductor exponents of hyperelliptic curves.

\begin{theorem}[{=\cite[Theorem 11.3]{M2D2}}]\label{thm:M2D2}
    Let ${C}/K: y^2=f(x)$, for square-free $f$, be a hyperelliptic curve over a finite extension of $\QQ_p$ with $p$ odd.
    \[n_{{C},\wild}=w_K(f).\]
\end{theorem}

\begin{remark}
    Note that $w_K(f)=w_K(\disc_y (y^2-f(x)))$ because these polynomials differ by multiplication by a constant, so Theorem \ref{thm:main2} is a direct generalisation of Theorem \ref{thm:M2D2}.

    Theorem \ref{thm:M2D2} is proved by observing that the wild inertia action on $2$-torsion of a hyperelliptic curve $y^2=f(x)$ is isomorphic to the wild inertia action on the roots of $f$. This relies on the explicit description of the $2$-torsion of such a curve, whilst Theorem \ref{thm:main2} does not rely on any explicit knowledge of torsion; we reduce to the explicit description of $2$-torsion on some auxiliary hyperelliptic discriminant curve after perturbation.
\end{remark}

Fixing some finite extension $K$ of $\QQ_p$, we first prove an important local constancy property of the quantity $w_K$.

\begin{lemma}\label{lem:Wconst}
    Let $g\in K[t]$ be $p^\text{th}$-power-free and suppose square-free $h\in K[t]$ is sufficiently close to $g$, in the sense that all of their coefficients are $p$-adically close. We have $w_K(h)=w_K(g)$.
\end{lemma}

\begin{proof}
    From the proof of \cite[Theorem 11.3]{M2D2}, when $g$ is square-free $w_K(g)$ is determined by the action of the wild inertia group $W_K$ on $\QQ_2[R]$, where $R$ is the set of roots of $g$ over $\bar{K}$.
    
    Suppose $g=\prod_ig_i^{e_i}$ for distinct irreducible $g_i$ and $e_i\le p-1$. Write $R_i$ for the roots of $g_i$ over $\bar{K}$ and $R_h$ for those of $h$. We claim that $\QQ_2[R_h]\cong\bigoplus_i\QQ_2[R_i]^{\oplus e_i}$ as $W_K$-representations, so $w_K(h)=w_K(g)$:

    Firstly recall that $W_K$ acts via a finite $p$-group, and so each orbit under the action of $W_K$ has size a power of $p$.
    
    Choose $r\in R_h$ and note that $r$ is close to a point in $r_0\in R_i$ for some $i$. In fact, there is a small neighbourhood of $r_0$ containing $e_i$ points. By continuity of the $W_K$-action, given such a neighbourhood of each $r_0\in R_i$, there is an orbit of these neighbourhoods mirroring the orbit of $r_0$. Supposing that the orbit of $r$ contains more than one point in any of these neighbourhoods leads to a contradiction: we would have
    \[|W_K\cdot r_0|<|W_K\cdot r|\le |W_K\cdot r_0|\cdot e_i \le |W_K\cdot r_0| \cdot (p-1),\] but this shows that $|W_K\cdot r_0|$ and $|W_K\cdot r|$ cannot both be powers of $p$.

    To conclude, we have shown that for each $W_K$-orbit of roots of $g_i$ we obtain $e_i$ distinct orbits of points in $R_h$, each defining a representation isomorphic to that of the original orbit.
\end{proof}

To prove our main theorem, we will need to distinguish between an affine curve $\mathcal{C}/K: f(x,y)=0$ and its smooth compactification $C$. The following lemma will help us reduce to the case of Theorem \ref{thm:Sn_main}.

\begin{lemma}\label{lem:perturb}
    Suppose $\mathcal{C}/K: f(x,y)=0$ is smooth away from infinity and is such that $g_C>0$. We can choose $\tilde{f}(x,y)$ with coefficients arbitrarily close to those of $f$ such that $\tilde{\mathcal{C}}/K: \tilde{f}(x,y)=0$ with compactification $\tilde{C}$ satisfies $g_C=g_{\tilde{C}}$, and $y: \tilde{\mathcal{C}}\to\mathbb{A}^1$ is simply branched.
\end{lemma}

\begin{proof}
    Write $n=\deg_x f$ and define ${f}_{\pmb{\eps}}(x,y)=f(x,y)+\sum_{i=0}^{n-1}\eps_i x^i$, where $\pmb{\eps}=(\eps_0,\hdots,\eps_{n-1})$ takes values in $\bar{K}^n$, noting that perturbing in this way will not affect the behaviour at infinity. 
    
    It suffices to show that $\disc_y\disc_xf_{\pmb{\eps}}(x,y)\ne0$ as a polynomial in the $\eps_i$. Indeed, in that case we can choose $\pmb{\eps}\in K^n$ small such that this polynomial does not vanish, which precisely means that $\disc_xf_{\pmb{\eps}}$ has no repeated root, so projection onto $y$ from the affine curve $f_{\pmb{\eps}}(x,y)=0$ is a simply branched cover of $\mathbb{A}^1$. To do this, we show that such an affine curve can be explicitly constructed with $\pmb{\eps}\in \bar{K}^n$. 
    
    Write $B$ for the set of branch points and $B'\subseteq B$ for the set of non-simple finite branch points of $C\to\Proj^1$:
        
    Choose a non-simple branch point $b\in B'$ and $\alpha$ a double root of $f(x,b)$. Take $f_{\pmb{\eps}}(x,y)=f(x,y)+\delta(x-\alpha)^2$. For $\delta$ sufficiently small, $f_{\pmb{\eps}}(x,y)=0$ defines a curve of the same genus for which projection onto $y$ is simply branched at $b$. This is because $f(x,b)+\delta(x-\alpha)^2=(x-\alpha)^2(g(x)+\delta)$ for some polynomial $g$, and now $g(x)+\delta$ is square-free so long as $\delta \ne -g(\beta)$ for $\beta$ any root of $g'$.

    Note that $\disc_x f_{\pmb{\eps}}(x,y)$ varies continuously with $\pmb{\eps}$ and hence so do the roots, which are the branch points counted with multiplicity $\sum_{P\in y^{-1}(b)}(e_P-1)$. For small enough $\delta$, then, there is a single simple branch point of $f_{\pmb{\eps}}(x,y)=0$ in a small neighbourhood of each finite simple branch point of $C$.
    
    Because the genera are the same by the degree-genus formula, the sum $\sum_{b\in B\setminus B'}\sum_{P\in y^{-1}(b)}(e_P-1)+\sum_{b\in B'}\sum_{P\in y^{-1}(b)}(e_P-1)$ is constant by Riemann--Hurwitz. The number of simple branch points has increased, so the quantity $\sum_{b\in B'}\sum_{P\in y^{-1}(b)}(e_P-1)$ has decreased, and iterating this process eventually gives the desired affine curve $f_{\pmb{\eps}}(x,y)=0$ defined over $\bar{K}$.
\end{proof}

Finally, we are able to show Theorem \ref{thm:main}.

\begin{theorem}\label{thm:main2}
    Let $\mathcal{C}:f(x,y)=0$ be a smooth affine curve over a finite extension $K$ of $\QQ_p$ with smooth compactification $C$. If $p>\deg_x f=n$, then
    \[n_{{C},\wild}=w_K(\disc_x f).\]
\end{theorem}

\begin{proof}
    Consider the map $\pi: C\to \Proj^1$ corresponding to $y: \mathcal{C}\to \mathbb{A}^1$. When this has $S_n$-Galois closure $X$, the curve $D=X/A_n$ is the hyperelliptic curve given by $D: \Delta^2 = \disc_x f$. The smoothness of $\mathcal{C}$ guarantees that
    \begin{equation}\label{eqn:disc}\tag{$**$}
        \disc_x f = u\cdot\prod_{t\in\Proj^1}\prod_{P\in\pi^{-1}(t)} (y-t)^{e_P-1}
    \end{equation}
    for some $u\in K^\times$, where $e_P$ is the ramification index at $P$ and $`y-\infty'$ is taken to be $1$. Note that without the assumption on smoothness, the right-hand side of \eqref{eqn:disc} need only divide the left. As in the proof of Theorem \ref{thm:1}, Lemma \ref{lem:curvinduct} gives $n_{C,\wild}=n_{\wild}(X^{\one+F})$ and $n_{D,\wild}=n_{\wild}(X^{\one + \eps})$. We have $n_{\wild}(X^{\one})=0$ and so, when $y:\mathcal{C}\to\mathbb{A}^1$ is simply branched, Theorem \ref{thm:Sn_main} ensures that $n_{C,\wild}=n_{D,\wild}$. The conclusion follows from Theorem \ref{thm:M2D2}, using that $\disc_x f$ is square-free. 

    To reduce to this case, we first perturb the equation $f(x,y)=0$ as in Lemma \ref{lem:perturb} to obtain $\tilde{C}: \tilde{f}(x,y)=0$ with corresponding cover $y:\tilde{\mathcal{C}}\to\mathbb{A}^1$ simply branched. If $\Gal(\tilde{f})\cong S_n$ as a polynomial over $\bar{K}(y)$, then we are done. We can ensure this is the case using the same perturbations as in Lemma \ref{lem:perturb}: if $\Gal(\tilde{f})\not\cong S_n$, then $\Gal(\tilde{f}_{\pmb{\eps}}(x,y_0))\lneq S_n$ for any choice of $y_0$, but now fix $y_0$ and choose $\pmb{\eps}'$ small with respect to $v_K$ such that $\Gal(\tilde{f}_{\pmb{\eps}'}(x,y_0))\cong S_n$, which is easily done. Because the roots of $\disc_x f_{\pmb{\eps}}$ vary continuously with $\pmb{\eps}$, this sufficiently small perturbation will not introduce any non-simple branching. 
    
    We conclude by Lemma \ref{lem:Wconst} and the local constancy of wild conductor exponents \cite[Theorem 5.1(1)]{Kisin1999}. The former applies because $\disc_x f$ is $n^\text{th}$-power-free by \eqref{eqn:disc} and the fact that $p>n$. The latter applies because perturbing in the described way yields an $\ell$-adic family of curves of generic genus $g_C$.
\end{proof}

\begin{corollary}\label{cor:super}
    Consider a superelliptic curve ${C}/K: y^n = f(x)$, $f$ square-free, over a finite extension $K$ of $\QQ_p$. If $p>n$, then
    \[n_{{C},\wild}=(n-1)\cdot w_K(f).\]
\end{corollary}

\begin{proof}
    Note that $\disc_y(y^n-f(x))$ is a constant multiple of $f^{n-1}$. After re-labelling $x$ and $y$, we are in the case described in Theorem \ref{thm:main2}.
\end{proof}

\appendix
\section{3-torsion of genus 2 curves}\label{app:A}

The following is \cite[Proposition 4.1]{DokDoris}, which is used in \emph{loc.\ cit.\ }to compute wild conductor exponents at $p=2$ of genus 2 curves.
\vspace{0.4cm}
\begin{center}
    \begin{minipage}{0.75\linewidth}
        \textit{Let ${C}/k$ be a genus 2 curve over a field of characteristic different from 2 and 3 with model $y^2=F(x)$. There is a one-to-one correspondence between the non-zero 3-torsion points of the Jacobian variety $J_{C}$ and tuples $(u_1,\hdots,u_7)\in \bar{k}^7$ such that
        \begin{equation}\label{eqn:system}
            F(x)=(u_4x^3+u_3x^2+u_2x+u_1)^2-u_7(x^2+u_6x+u_5)^3.
        \end{equation} 
        Moreover, this correspondence preserves the action of the absolute Galois group $G_k$.}
    \end{minipage}
\end{center}
\vspace{0.4cm}

It turns out that this sometimes fails to pick up all $3$-torsion points.

\begin{example}\label{ex:counter}
    Consider the curve with LMFDB \cite{lmfdb} label \href{https://www.lmfdb.org/Genus2Curve/Q/1744/a/1744/1}{1744.a.1744.1} which has model
    \[{C}/\QQ: \hspace{0.5cm} y^2=F(x)=(x^3+x)(x^3+x+4),\]
    so $J_{C}(\QQ)\cong\ZZ/6\ZZ$. According to the above, we should expect two rational solutions to the system of equations described, but a check using Magma proves that there are none. Indeed, there are only 78 solutions to this system of equations over $\bar{\QQ}$.
\end{example}

A corrected version is as follows.

\begin{proposition}\label{prop:genus2tors}
    Let $C/k$ be a genus 2 curve over a field of characteristic different from 2 and 3 with model $y^2=F(x)$. There is a one-to-one correspondence between the non-zero 3-torsion points of $J_{C}$ and the union of the sets:
    \begin{enumerate}
        \item Tuples $(u_1,u_2,u_3,u_4,u_5,u_6,u_7)\in \bar{k}^7$ such that
        \[F(x)=(u_4x^3+u_3x^2+u_2x+u_1)^2-u_7(x^2+u_6x+u_5)^3;\]
        \item Tuples $(v_1,v_2,v_3,v_4,v_5,v_6)\in \bar{k}^6$ such that
        \[F(x)=(v_4x^3+v_3x^2+v_2x+v_1)^2-v_6(x+v_5)^3;\]
        \item Tuples $(w_1,w_2,w_3,w_4,w_5)\in \bar{k}^5$ such that
        \[F(x)=(w_4x^3+w_3x^2+w_2x+w_1)^2-w_5.\]
    \end{enumerate}
    Moreover, this correspondence preserves the action of $G_k$.
\end{proposition}

\begin{proof}
    Suppose $\mathcal{D}$ is a divisor on ${C}$ such that $3\mathcal{D}$ is principal. Explicitly,
\[ \mathcal{D} = (P_1)+(P_2)-(\infty_1)-(\infty_2), \hspace{1cm} P_i=(X_i,Y_i) \text{ or } P_i\in\{\infty_{1,2}\}\]
for some points $P_i$ on ${C}$ and where $\infty_{1,2}$ are the not-necessarily-distinct points at infinity. There exists some rational function $g$ with $\div(g)=3D$.
Then $g$ is in the Riemann--Roch space ${L}(3(\infty_1)+3(\infty_2))=\langle1,x,x^2,x^3,y\rangle$. It is easy to see that the coefficient of $y$ must be non-zero and so
\[g=y+a_3x^3+a_2x^2+a_1x+a_0\]
after a suitable re-scaling.

Taking norms from $k({C})$ to $k(x)$ yields a function $(a_3x^3+a_2x^2+a_1x+a_0)^2-F(x)$ on $\Proj^1$, for some $a_i$, with divisor $3(X_1)+3(X_2)-6(\infty)$ a cube. Provided ${C}$ does not admit a degree 3 map to $\Proj^1$, we have, for some $b_j$,
\[(a_3x^3+a_2x^2+a_1x+a_0)^2-F(x)=b_2(x^2+b_1x+b_0)^3.\]

We needed this assumption to avoid the case that $P_i=\infty_j$ for some $i,j$. In this case we have, without loss of generality,
\[\mathcal{D}=(P_1)-(\infty_1), \hspace{1cm} \div(g)=3(P_1)-3(\infty_1)\] 
and $g: {C}\to \Proj^1$ is a degree 3 map. 

Now the norm map yields a function as above which is still a cube, but now has divisor $3(X_1)-3(\infty)$. Hence we conclude that
\[(a_3x^3+a_2x^2+a_1x+a_0)^2-F(x)=b_1(x+b_0)^3,\]
for some $b_j$, unless $P_1=\infty_2$. In this final case the norm map yields a constant function and so we conclude that
\[(a_3x^3+a_2x^2+a_1x+a_0)^2-F(x)=b,\]
for some $b$.
\end{proof}

Note that the original result is correct if one starts with $\deg F$ being odd, with such a model existing whenever $C$ has a rational point. The troublesome cases for which it may fail are those curves admitting degree 3 covers $\pi:C\to\Proj^1$ with associated `discriminant curve' of genus at most $1$. By this we refer (up to quadratic twist) to the curve
\[ D: \hspace{0.2cm} y^2 = \prod_{t\in\Proj^1}\prod_{P\in\pi^{-1}(t)} (x-t)^{e_P-1}, \]
as in \eqref{eqn:disc} from the proof of Theorem \ref{thm:main2}. Examples of curves $C$ admitting such covers are mentioned in \cite[Example 2]{BFT}. That these are the troublesome cases can be seen from the following lemma.

\begin{lemma}\label{lem:3tors}
    Let ${C}/k$ be a genus 2 curve. There is a divisor $\mathcal{D}$ on ${C}$ of the form
    \[\mathcal{D}=(P)-(Q)\]
    corresponding to a non-zero element of $J_{C}[3]$ if and only if $C$ admits a degree 3 cover of $\Proj^1$ which is non-simply branched above at least two points, or equivalently such that the corresponding discriminant curve $D$ has genus strictly less than 2. 
\end{lemma}

\begin{proof}
    For the `if' direction, suppose that we are given $\pi:C\to\Proj^1$ of degree $3$ such that the corresponding discriminant curve $D$ has genus less than $2$. Writing $R_2$ for the number of simple branch points and $R_3$ for the number of non-simple branch points of $\pi$, we have $8=R_2+2R_3$ and $g_D=-1+R_2/2$ by Riemann--Hurwitz. This implies that $\pi$ is non-simply branched above at least two points, say $t_1$ and $t_2$. After a translation sending $t_1$ to $0$ and $t_2$ to $\infty$, the divisor $\mathcal{D}=\div(\pi)/3$ of the claimed form lies in $J_C[3]$. 

    For the `only if' direction, suppose $\mathcal{D}=(P)-(Q)$ is in $J_C[3]$. Then there exists $\pi:C\to\Proj^1$ with $\div(\pi)=3(P)-3(Q)$. This says that $\pi$ has degree $3$ and is non-simply branched above $0$ and above $\infty$.

    The equivalence of the branching condition and the condition on the genus of $D$ follows from a Riemann--Hurwitz calculation.
\end{proof}

\bibliographystyle{plain}
\bibliography{main.bib}
\end{document}